\documentclass[12pt]{iopart}

\usepackage{iopams}
\usepackage{amsthm}
\usepackage{graphicx}
\usepackage{psfrag}
\newcommand{\R}{{\mathbb{R}}}

\newtheorem{thm}{Theorem}[section]
\newtheorem{cor}[thm]{Corollary}
\newtheorem{lem}[thm]{Lemma}
\newtheorem{prop}[thm]{Proposition}
\theoremstyle{definition}

\newtheorem{defn}[thm]{Definition}
\theoremstyle{remark}

\begin{document}

\title[Uniqueness of models in persistence
homology]{Uniqueness of models in persistent
homology: the case of curves}

\author{P Frosini$^{1,3}$, C Landi$^{2,3}$}

\address{$^1$ Dipartimento di Matematica, Universit\`a di Bologna, Italy}
\address{$^2$ Dipartimento di Scienze e Metodi dell'Ingegneria, Universit\`a di Modena e Reggio Emilia, Italy}
\address{$^3$ ARCES, Universit\`a di Bologna, Italy}
\ead{frosini@dm.unibo.it, clandi@unimore.it}

\begin{abstract}
We consider generic  curves in $\R^2$, i.e. generic $C^1$ functions $f:S^1\to\R^2$. We analyze these curves through the persistent homology groups of a filtration induced on $S^1$ by  $f$. In particular, we consider the question whether these  persistent homology groups  uniquely characterize $f$, at least up to re-parameterizations of $S^1$. We give a partially positive  answer to this question. More precisely, we prove that $f=g\circ h$, where $h:S^1\rightarrow S^1$ is a $C^1$-diffeomorphism, if and only if the persistent homology groups of $s\circ f$ and $s\circ g$ coincide, for every  $s$ belonging to the group $\Sigma_2$ generated by reflections in the coordinate axes.  Moreover, for a smaller set of generic functions, we show that $f$ and $g$ are close to each other in the max-norm (up to re-parameterizations) if and only if, for every  $s\in \Sigma_2$, the persistent Betti numbers functions of $s\circ f$ and $s\circ g$ are close to each other, with respect to a suitable distance.
\end{abstract}

\ams{55N35,  53A04, 68U05}
\vspace{2pc}
{\it Keywords\/}: Persistent Betti numbers, reflections, matching distance

\submitto{\IP}

\section{Introduction}
Persistent homology is a widely studied tool in Topological Data Analysis. It is
based on investigating topological spaces by growing a space (i.e., the data to be
studied) incrementally, and by analyzing the topological changes that occur during
this growth. The occurrence and placement of topological events (e.g., creation,
merging, cancellation of the connected components of the lower level sets) within
the history of this growth describe the essential geometrical properties of the
data. Persistent homology aims to define a scale of the relevance of these
topological events, where the longer the lifetime of a feature produced by a topological event, the more significant the event.

An area of application of the persistent homology theory in TDA is shape
description \cite{VeUr*93,CoZo*04}. In this setting the studied topological space $X$ represents the object
whose shape is under study, and its shape is analyzed my means of a vector-valued
function $f$ defined on it. This function corresponds to measurements on the data
depending on the shape properties of interest (e.g.,
elongation, bumpiness, curvature, and so on). This function is then used to filter
the space by lower level sets. The persistent homology of this filtration gives
insights on the shape of $X$ as seen through $f$. In particular, persistence diagrams,
i.e. multisets of points of the plane encoding the rank of persistence homology
groups, constitute a shape descriptor, or a signature, of $(X,f)$
(cf. \cite{CoEdHa07}).

We recall that while an object representation (either pixel- or
vector-based) contains enough information to reconstruct (an approximation to) the
object, a description only contains enough information to identify an object as a
member of some class, usually by means of a dissimilarity measure. The
representation of an object is thus more detailed and accurate than a description,
whereas the description is more concise and conveys an elaborate and composite view
of the object class.

In the illustrated framework, two objects $X$ and $Y$ belong to
the same class if they behave in a similar way with respect to the chosen shape
property represented by the continuous functions $f:X\to \R^k$ and $g:Y\to \R^k$. More
formally, $(X,f)$ and $(Y,g)$ belong to the same object class if and only if there is a
homeomorphism $h:X\to Y$ such that $f=g\circ h$. This condition immediately implies that
$(X,f)$ and $(Y,g)$ have the same persistent homology groups, while it is easy to give examples showing that in general this implication cannot be reversed.

Until now, research has been mainly focused on direct problems, such as, given $X$
and $f$, computing persistence diagrams, establishing stability properties of
persistence diagrams, choosing functions in order to impose desired invariance
properties.

As far as inverse problems  are concerned in this setting, there have
been some attempts to study the problem of existence of models in persistence
homology. For example, confining our attention to the 0th homology degree and $k=1$,
it is known under which conditions a multiset of points of the plane is the
persistent diagram of some space $X$ endowed with some function $f:X\to \R$.
Furthermore, it is possible to explicitly construct a space $X$ and a function $f$
having a prescribed persistence diagram, i.e. a model for a given persistence
diagram. For more details about this line of research we refer the reader to \cite{dAFrLa10}. Moreover, a realization result for finite persistence models is stated in \cite{CaZo09}.

In this paper we will tackle the inverse problem related to the uniqueness of the
model. What does uniqueness mean in this setting?
It means that there is exactly
one model with given persistent homology groups up to the equivalence relation for which
``$(X,f)$ and $(Y,g)$ are equivalent if and only if there is a homeomorphism $h:X\to Y$ such
that $f=g\circ h$''.

We underline that different formulations of uniqueness would give rise to  either impossible or trivial problems. Indeed, in general, it is false that if $(X,f)$ and $(X,g)$ have the same persistent homology groups then necessarily $f=g$. On the other hand, it is easy to see that, for any  space $X\subseteq \R^2$,  taking the function $f:X\to \R^4$ defined by $f(x,y)=(x,-x,y,-y)$, the persistent homology groups of $(X,f)$ uniquely determine $X$. However, this would not be a satisfactory solution of the uniqueness problem, in first place because  it would work  with only one prescribed function $f$; in second place because in pattern recognition the focus is generally on parametrization-independent shape comparison methods (cf, e.g., \cite{MiMu06}).

Our uniqueness problem is clearly strictly related to the decision
problem in shape matching, that is, given two patterns, deciding whether there
exists a transformation taking one pattern to the other pattern. Rephrased
differently, we wish to study to which an extent persistent homology can give rise
to complete shape invariants.

We also observe that the problem of deciding whether two functions are obtained one from the other by a re-parameterization is also strictly related to the concept of natural pseudo-distance between the pairs $(X,f)$ and $(Y,g)$ (we refer the interested reader to \cite{FrMu99,DoFr04,DoFr07,DoFr09}).

As the reader can guess, this subject is not simple. We  know well that homology is not sufficient to reconstruct a manifold up to diffeomorphisms, and clearly also persistent homology has analogous limitations. Indeed, several examples in this paper prove that some kind of indeterminacy and non-uniqueness is unavoidable, also in the case of curves, i.e. when $X=S^1$. However, in this paper we can show that the situation is not so negative as it could appear at a first glance. In particular, we shall prove that, at least in the  case of generic curves, in the differentiable category,
persistent homology provides sufficient information to identify the studied function up to diffeomorphisms of $S^1$ (Theorem \ref{th1}).   Moreover, we show that, under mild assumptions,  the proximity between persistent Betti numbers functions of two curves implies proximity between the curves themselves (Theorem \ref{th2}).

\section{Notations and basic definitions}

In this paper we confine ourselves to study the uniqueness of models  when $X$ is a one-dimensional  manifold without boundary,  in the $C^1$-differentiable case. Since any such curve $X$ is diffeomorphic to the standard circle $S^1$, choosing a fixed diffeomorphism from $X$ to $S^1$, we can confine our study  to the case $X=S^1$.  Therefore, our problem can be restated as follows: is it true that, given two functions $f$ and $g$ on $S^1$, the associated  persistent homology groups coincide if and only if $g$ is a re-parameterization of $f$?

In order to deal with our uniqueness problem we will use  only the rank of $0$th persistent
homology groups but in a bi-dimensional setting, that is to say  bi-dimensional size functions \cite{BiCeFrGiLa08}. We recall here their basic definitions, as a particular case of the more general theory of multidimensional persistence (cf. \cite{CaZo09}).

For any point $u=(u_1,u_2)\in \R^2$, we denote by $D^u$ the set $\{w=(w_1,w_2)\in \R^2: w_1\le u_1 \wedge w_2\le u_2\}$. For any continuous function
$f=(f_1,f_2):S^1\to\R^2$ and $u\in \R^2$, we can consider the bi-filtration of $S^1$ given by  $\{f^{-1}(D^u)\}_{u\in \R^2}$. The symbol $\Delta^+$ will denote the open
set $\{( u,v)\in\R^2\times\R^2:  u_1<v_1\wedge  u_2<v_2\}$.

\begin{defn}
For any $(u , v)\in \Delta^+$, the {\em $0$th bi-dimensional persistent homology group} of $f$ at $( u , v)$ is the group $$H_0^{u,v}(f)=\mathrm{im\,}H_0(f^{-1}(D^u)\hookrightarrow f^{-1}(D^v)),$$
where $f^{-1}(D^u)\hookrightarrow f^{-1}(D^v)$ is the  inclusion map.
\end{defn}
Here the considered homology theory is the \v{C}ech one, with real coefficients. In plain words, the rank of $H_0^{u,v}(f)$ is equal to the number of connected components of $f^{-1}(D^v)$ that contain at least one point in $f^{-1}(D^u)$.  
We remark that,  since $S^1$ is a compact manifold, for any continuous function $f: S^1\rightarrow \R^2$, its persistent homology groups are  finitely generated \cite{CaLa11}. Hence, their rank, also known as a {\em persistent Betti number}, is finite.

 In order to make our treatment more readable, we shall use the same symbol $\theta$ to denote both each point of $S^1$, and the local parameterization of $S^1$ that we shall use in derivatives. This requires a little abuse of notation since, rigorously speaking, we should denote the points in $S^1$ by equivalence classes of angles $\theta\in \R$ (equivalent up to multiples of $2\pi$). We also assume that $\theta$ is counterclockwise increasing.
 
\section{Generic assumptions on functions}

We begin by presenting some  negative examples. In these examples we add more and more assumptions showing that without those assumptions the model uniqueness fails. We will end with two conditions (C1), (C2) on the functions $f,g$ defined on $S^1$ that, as we will show in the next section, are sufficient to guarantee uniqueness. We end this section showing that the set of functions satisfying conditions (C1), (C2) is dense in $C^1(S^1,\R^2)$. In other words, we will prove the model uniqueness for a generic set of functions defined on simple curves. Let us  remark that, although we are assuming that $X$ is a simple curve (indeed diffeomorphic to $S^1$), the considered functions $f$ defined on $X$ can give rise to multiple points.

The first example, illustrated in Figure \ref{tri}, shows two simple closed curves $X$ (left) and $Y$ (right) endowed with continuous functions $f:X\to \R$ and $g: Y\to \R$, such that the persistent homology groups of $f$ and $g$ coincide at every $(u,v)\in \Delta ^+$ (center). However there does not exist any $C^1$-diffeomorphism $h: X\to Y$ such $f=g\circ h$. Indeed, a $C^1$-diffeomorphism $h: X\to Y$ such that $f=g\circ h$ should take critical points of $f$ into critical points of $g$ preserving their values and adjacencies. This is clearly impossible.

\begin{figure}[h]
\begin{center}
\psfrag{X}{$X$}\psfrag{Y}{$Y$}
\psfrag{f1}{$f$}\psfrag{f2}{$g$}
\psfrag{u}{$u$}\psfrag{v}{$v$}
\psfrag{a}{$a$}\psfrag{b}{$b$}\psfrag{c}{$c$}\psfrag{d}{$d$}\psfrag{e}{$e$}
\psfrag{0}{$0$}\psfrag{1}{$1$}\psfrag{2}{$2$}\psfrag{3}{$3$}
\includegraphics[width=14cm]{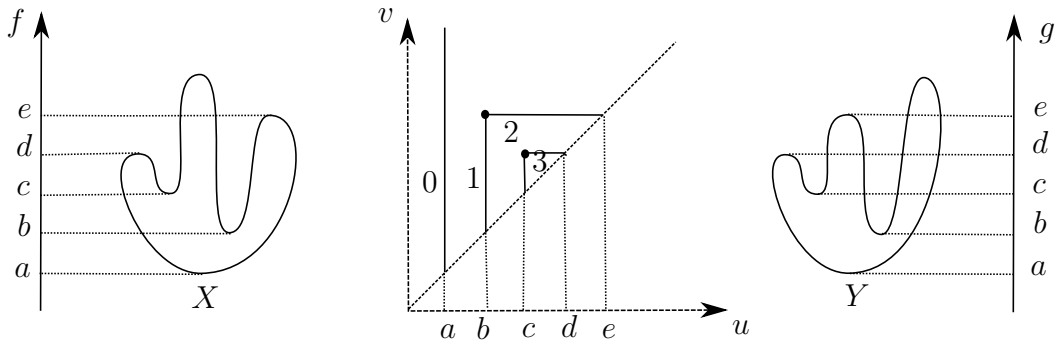}
\caption{The curves $X$ (left) and  $Y$ (right), endowed with the continuous functions $f:X\to \R$ and $g: Y\to \R$ respectively,  cannot be distinguished by the persistent homology groups of $f$ and $g$, as  their ranks coincide everywhere in $\Delta^+$ (center). }\label{tri}
\end{center}
\end{figure}

It is interesting to note that also changing the functions $f$ and $g$ into their opposite, the closed curves $X$, $Y$ cannot be distinguished. Indeed, also the persistent homology groups of $-f$ and $-g$ coincide at every $(u,v)\in \Delta ^+$. Obviously, there does not exist any $C^1$-diffeomorphism $h: X\to Y$ such that $-f=-g\circ h$ (see Figure \ref{tri2}).

\begin{figure}[h]
\begin{center}
\psfrag{X}{$X$}\psfrag{Y}{$Y$}
\psfrag{f1}{$-f$}\psfrag{f2}{$-g$}
\psfrag{u}{$u$}\psfrag{v}{$v$}
\psfrag{m}{$m$}\psfrag{-b}{$-b$}\psfrag{-c}{$-c$}\psfrag{-d}{$-d$}\psfrag{-e}{$-e$}
\psfrag{0}{$0$}\psfrag{1}{$1$}\psfrag{2}{$2$}\psfrag{3}{$3$}
\includegraphics[width=14cm]{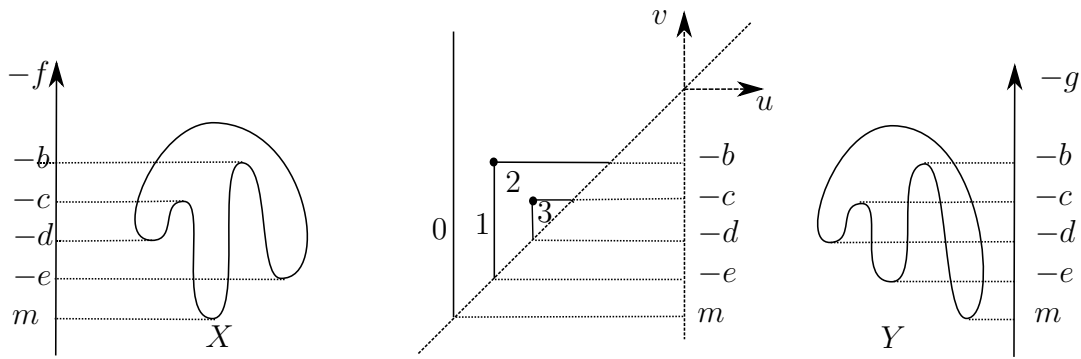}
\caption{The curves $X$, $Y$ endowed with the continuous functions $-f:X\to \R$ and $-g: Y\to \R$ cannot be distinguished by the persistent homology groups of $-f$ and $-g$.}\label{tri2}
\end{center}
\end{figure}

These two examples suggest us to consider vector-valued rather than scalar functions. A similar (but slightly more complicated) example is exhibited in \cite{CaFePo01}.

The second example, illustrated in Figure \ref{ex2}, shows the image of two functions $f,g:S^1\to \R^2$ such that   the persistent homology groups of $ f$ and $ g$ coincide at every $( u, v)\in \Delta ^+$, as can be checked by a direct computation. However, there does not exist any $C^1$-diffeomorphism $h: S^1\to S^1$ such that $f=g\circ h$ because $f(S^1)\ne g(S^1)$.

\begin{figure}[h]
\begin{center}
\psfrag{f1}{$f_1$}\psfrag{f2}{$f_2$}\psfrag{g1}{$g_1$}\psfrag{g2}{$g_2$}
\includegraphics[width=10cm]{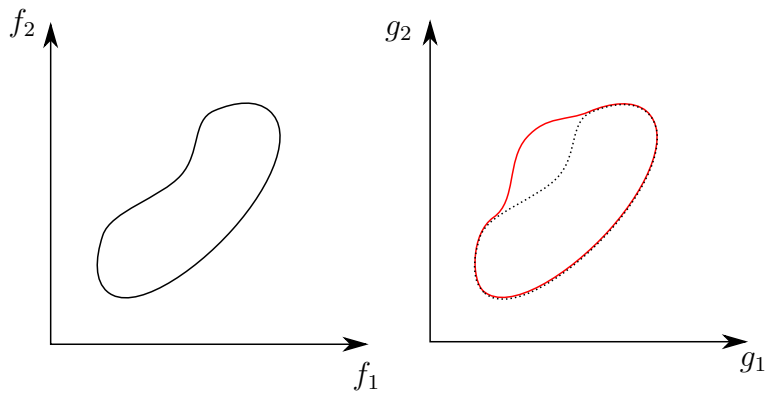}
\caption{The curves $f=(f_1,f_2):S^1\to \R^2$ and $g=(g_1,g_2): S^1\to \R^2$ cannot be distinguished by their persistent homology groups.}\label{ex2}
\end{center}
\end{figure}

This example suggests us that it is not enough to require that $H_0^{ u, v}(f)=H_0^{ u, v}(g)$ for every $( u, v)\in \Delta^+$, but we should take stronger assumptions such as that also $H_0^{ u, v}(s\circ f)=H_0^{ u, v}(s\circ g)$ for every $( u, v)\in \Delta^+$, and every $s:\R^2\to \R^2$ obtained via composition of reflections with respect to the coordinate axes.

The last example shows that, even under these stronger assumptions, the model uniqueness  fails, and suggests us to add the assumption that there are no two distinct points $\theta_1,\theta_2$ in $S^1$ such that $f(\theta_1)=f(\theta_2)$ and  $\mathrm{im}\, d_{\theta_1}f = \mathrm{im}\, d_{\theta_2}f$.
Indeed, the curves $f=(f_1,f_2):S^1\to \R^2$ and $g=(g_1,g_2): S^1\to \R^2$ illustrated in Figure \ref{ex3} cannot be distinguished by their persistent homology groups, as can be seen by direct computations.
However, no $C^1$-diffeomorphism $h: S^1\to S^1$ exists such that $f=g\circ h$. Indeed, if it were the case,  $h$ should take the two points $\theta_1$, $\theta_2$ where $f_2$ takes its minimum into the two points $\tilde\theta_1$, $\tilde\theta_2$ were $g_2$ takes its minimum, and an arc  between $\theta_1$ and $\theta_2$ into an arc between $\tilde\theta_1$ and $\tilde \theta_2$.   It is easy to see that, for any possible choice of these arcs, the image through  $g$ would not coincide with that through $f$.

\begin{figure}[h]
\begin{center}
\begin{tabular}{|c|c|}
\hline
 & \\
\begin{minipage}{0.5\linewidth}\centering
\psfrag{f1}{$f_1$}\psfrag{f2}{$f_2$}\psfrag{g1}{$g_1$}\psfrag{g2}{$g_2$}
\includegraphics[width=7cm]{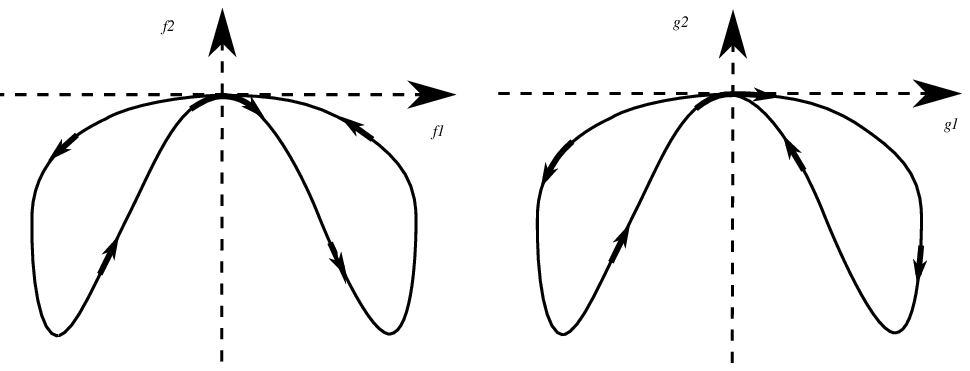} \end{minipage} &
\begin{minipage}{0.5\linewidth}\centering
\psfrag{f1}{$-f_1$}\psfrag{f2}{$f_2$}\psfrag{g1}{$-g_1$}\psfrag{g2}{$g_2$}
\includegraphics[width=7cm]{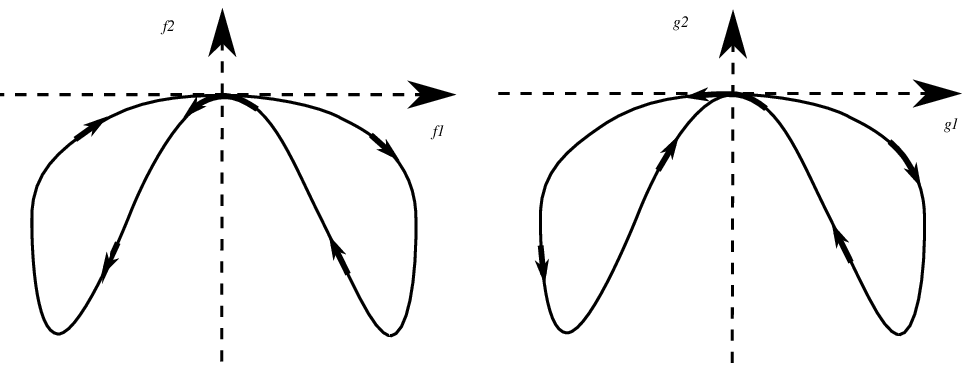}\end{minipage}\\
 & \\
\hline
 & \\
\begin{minipage}{0.5\linewidth}\centering
\psfrag{f1}{$f_1$}\psfrag{f2}{$-f_2$}\psfrag{g1}{$g_1$}\psfrag{g2}{$-g_2$}
\includegraphics[width=7cm]{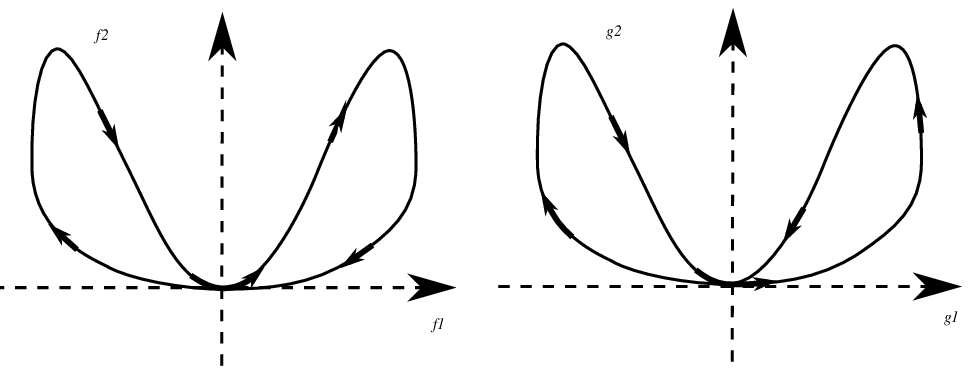} \end{minipage}&
\begin{minipage}{0.5\linewidth}\centering
\psfrag{f1}{$-f_1$}\psfrag{f2}{$-f_2$}\psfrag{g1}{$-g_1$}\psfrag{g2}{$-g_2$}
\includegraphics[width=7cm]{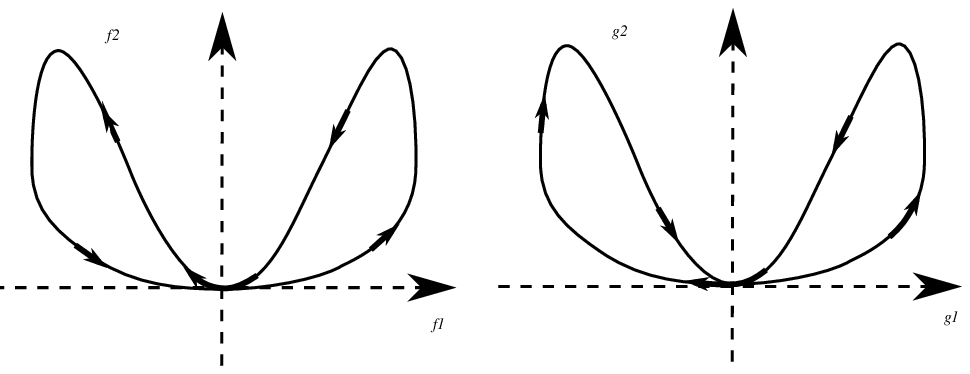}\end{minipage}\\
& \\ \hline
\end{tabular}
\caption{The curves $f=(f_1,f_2):S^1\to \R^2$ and $g=(g_1,g_2): S^1\to \R^2$ cannot be distinguished by their persistent homology groups. Analogously for the pairs of curves $s_1\circ f=(-f_1,f_2)$ and $s_1\circ g=(-g_1,g_2)$, $s_2\circ f=(f_1,-f_2)$ and $s_2\circ g=(g_1,-g_2)$, $s_2\circ s_1\circ f=(-f_1,-f_2)$ and $s_2\circ s_1\circ g=(-g_1,-g_2)$. }\label{ex3}
\end{center}
\end{figure}

These examples lead us to study the uniqueness problem taking functions as in the following definition, and assuming to have information also on the persistent homology groups of the functions obtainable by composition with reflections. The choice of confining ourselves to the following set of functions is not very restrictive since it is a dense set.

\begin{defn}\label{d1}
A function $f:S^1\rightarrow \R^2$ will be called {\em generic} if it is $C^1$ and the following properties hold:
\begin{enumerate}
\item[(C1)] $f$ is an immersion, i.e. $d_{\theta}f$ has rank equal to one for every $\theta\in S^1$;
\item[(C2)] $f(S^1)$ has at most a finite number of multiple points, all of them are double points and they are {\em clean}, i.e. $f(\theta_1)=f(\theta_2)$ and $\mathrm{im}\ d_{\theta_1}f= \mathrm{im}\ d_{\theta_2}f$ imply $\theta_1=\theta_2$, for every $\theta_1,\theta_2\in S^1$.
\end{enumerate}
\end{defn}

\begin{prop}\label{generic}
The set of generic functions is dense and open in $C^1(S^1,\R^2)$.
\end{prop}

\begin{proof}
Let us see that generic functions are dense in  $C^1(S^1,\R^2)$.
First of all $C^2(S^1,\R^2)$ is dense in $C^1(S^1,\R^2)$ (cf. \cite{Hi76}). The set of $C^2$-immersions of a manifold of dimension $1$ into a manifold of dimension $2$ is residual as an application of the  Jet Transversality Theorem, and, by the Multijet Transversality Theorem, also the set of functions satisfying (C2) is residual (cf. \cite{De00, Hi76}). Thus, the sets of $C^2$-functions separately satisfying conditions (C1) and (C2) are  residual  in $C^2(S^1,\R^2)$.  Moreover, any intersection of residual sets is still residual, and hence dense.  As a consequence, arbitrarily close to any $C^1$-function we can find a $C^2$-immersion (in particular of class $C^1$) satisfying (C2). So the set of generic functions is dense in $C^1(S^1,\R^2)$.

Finally, since the set of $C^1$-immersions is open in $C^1(S^1,\R^2)$ and the set of immersions with clean double points is open in the space of $C^1$-immersions (cf. \cite{Hi76}), the set of generic functions is  open in $C^1(S^1,\R^2)$.
\end{proof}

\section{Main results}

In this section we present the main results of this paper. Theorem \ref{th1} answers affirmatively to the uniqueness problem for generic functions and assuming  information is available also on the persistent homology groups of the functions obtainable by composition with reflections. Theorem \ref{th2} extends the previous result to the case when data are perturbed. Roughly speaking, it states that if two functions $f$ and $g$, together with their composition with reflections, give rise to close persistent Betti numbers,  then $f$ and $g$ are close to each other (in both cases closeness is meant with respect to a suitable distance).

Let $s_i:\mathbb{R}^2\to\mathbb{R}^2$, with $i=1,2$, be the reflections with respect to the coordinate axes: $s_1(x_1, x_2)=(-x_1,x_2)$, $s_2(x_1, x_2)=(x_1,-x_2)$. Let $\Sigma_2$ be the set of functions obtainable through finite composition of the reflections $s_1,s_2$ (obviously, $id\in \Sigma_2$).

\begin{thm}\label{th1}
Let $f,g:S^1\to \R^2$ be generic functions. If $H_0^{ u, v}(s\circ f)= H_0^{ u, v}(s\circ g)$ for every $( u, v)\in \Delta^+$ and every $s\in \Sigma_2$, then there exists a $C^1$-diffeomorphism $h:S^1\to S^1$ such that $g\circ h=f$.
\end{thm}

\begin{proof}
Since $f$ is generic, in particular it is an immersion. So, for each point $\theta\in S^1$, we can consider an open neighborhood $U$ of $\theta$ in $S^1$, such that $f_{|U}$ is a $C^1$-diffeomorphism onto its image. The line $l_\theta$ orthogonal to the  line tangent to $f(U)\subset \R^2$ at $f(\theta)$ is independent of the neighborhood $U$, and depends only on the point $\theta$.
Let $N_f$ be the set of all points $\theta$ of $S^1$ such that $l_\theta$ is not parallel to the coordinate axes of $\R^2$. The set $N_g$ is defined analogously. 

First of all, we shall prove that $f(\bar N_f)\subseteq g(S^1)$, where $\bar N_f$ is the closure of $N_f$ in $S^1$. We observe that since $f$ is  $C^1$, $N_f$ is non-empty and open in $S^1$.

By contradiction, let us assume that there exists a point $ u \in f(\bar N_f)\setminus g(S^1)$. For every $\epsilon >0$, a point $ u'\in f(N_f)\setminus g(S^1)$ exists such that $\| u- u'\|_2<\epsilon$, and $u'$ is not a double point of $f(S^1)$. Indeed, $g(S^1)$ is a closed set and $f$ is generic (in particular, $f(S^1)$ has at most a finite number of multiple points).

We set $\theta'=f^{-1}( u')$. Clearly $\theta'\in N_f$, implying that there is an $s'\in \Sigma_2$ such that $s'(l_{\theta'})$ has a unit direction  vector $w'=(w_1',w_2')$ with $w_1'>0$ and $w_2'>0$ (see Fig.~\ref{w}).

\begin{figure}[h]
\begin{center}
\includegraphics[width=10cm]{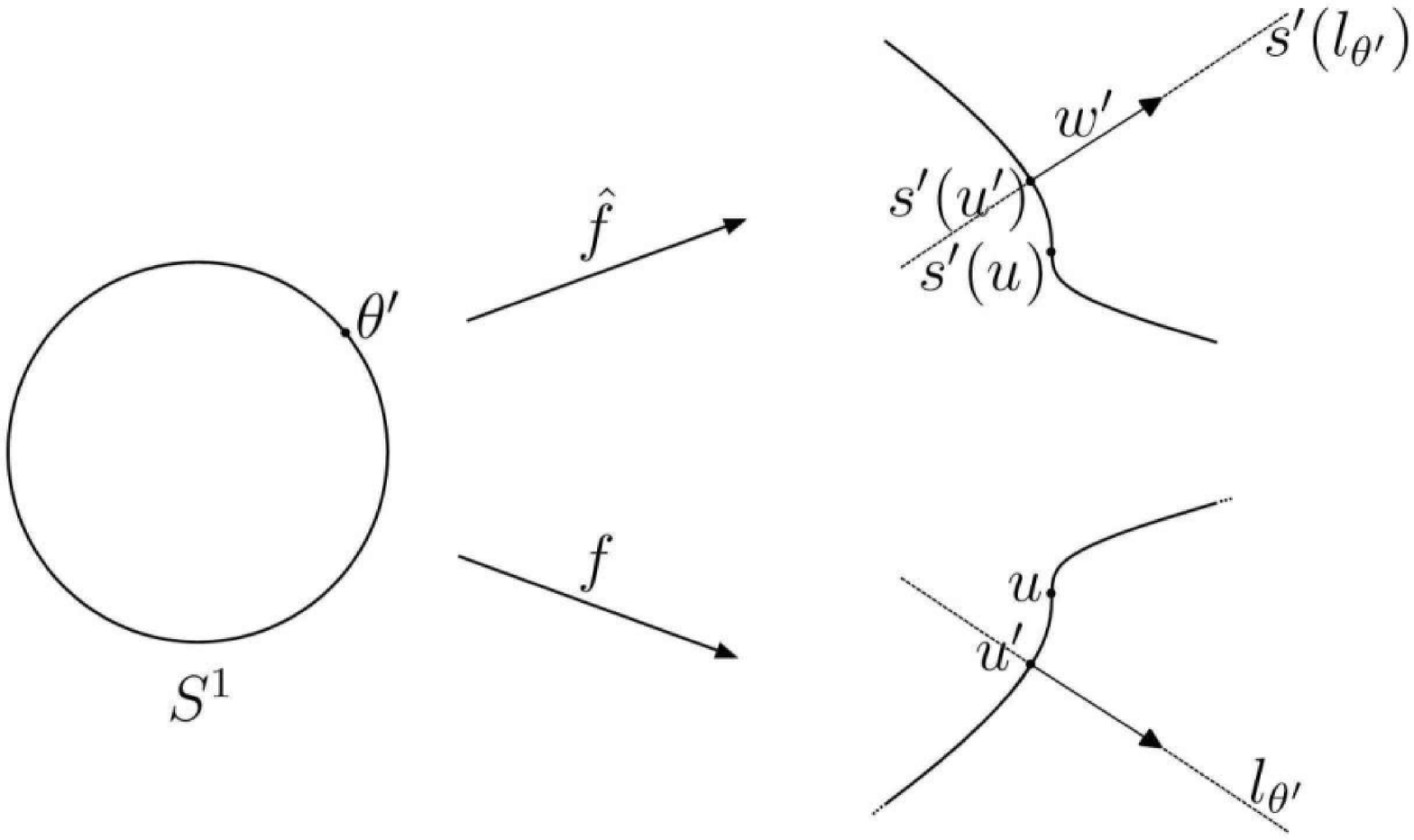}
\caption{The construction in the proof of Theorem~\ref{th1}.}\label{w}
\end{center}
\end{figure}

\begin{figure}[h]
\begin{center}
\psfrag{G}{$\hat f(\gamma)$}\psfrag{R}{$R$}\psfrag{a}{$a$}\psfrag{b}{$b$}\psfrag{c}{$c$}\psfrag{d}{$d$}\psfrag{v}{$v$}\psfrag{s}{$s'(u')$}\psfrag{s'}{$s'(u)$}
\includegraphics[height=4cm]{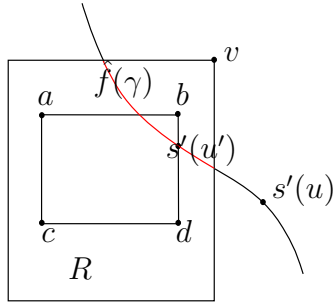}
\caption{The position of the points $a,b,c,d$ used in the proof of Theorem~\ref{th1}, with $\hat f$ top-right transversal to $R$.}\label{quad}
\end{center}
\end{figure}

We set $\hat f=s'\circ f$ and $\hat g=s'\circ g$. Because of this choice of $s'$, and since $s'(u')\in\hat f(S^1)\setminus \hat g(S^1)$, an open rectangle $R$ in $\mathbb{R}^2$ exists, with sides parallel to the coordinate axes, such that (see Fig.~\ref{quad})
\begin{enumerate}
  \item The set $\{\theta\in S^1:\hat f(\theta)\in R\}$ is an open connected arc $\gamma=\stackrel{\curvearrowright}{\theta_1\theta_2}$ (clockwise oriented) such that either $\hat f_1$ is increasing and $\hat f_2$ is decreasing on $\gamma$, or $\hat f_1$ is decreasing and $\hat f_2$ is increasing on $\gamma$ (as a consequence, the endpoints of $\hat f(\gamma)$ belong to $\partial R$);
  \item $R$ does not meet $\hat g(S^1)$.
\end{enumerate}
Indeed,  since $w_1'>0$ and $w_2'>0$, any non-vanishing tangent vector to $\hat f(\gamma)$ must have at least one strictly positive component.
In the following, when property (i) holds with respect to a rectangle $R$ we shall say that $\hat f$ is \emph{top-right transversal to $R$}.

Let $ v=(v_1,v_2)$ be the top-right vertex of $R$.
By property (i),  in $R\setminus\hat f(\gamma)$ we can take four points $a=(a_1,a_2)$, $b=(b_1,b_2)$, $c=(c_1,c_2)$, $d=(d_1,d_2)$ with $a_1=c_1<b_1=d_1<v_1$ and $c_2=d_2<a_2=b_2<v_2$, such that $a,c,d$ do not belong to the connected component of $b$ in $R\setminus\hat f(\gamma)$ and the segment $\overline{db}$ contains the point $s'( u')$.

We claim that
\begin{equation}\label{eq1}
    \mathrm{rk}\ H_0^{b, v}(\hat f)-\mathrm{rk}\ H_0^{d, v}(\hat f)-\mathrm{rk}\ H_0^{a, v}(\hat f)+\mathrm{rk}\ H_0^{c, v}(\hat f)=1.
\end{equation}

Indeed, with respect to the $C^1$ function $\hat f$, the number of connected components that are ``born'' between $c$ and $a$ and still not merged at $ v$ is one less than  the number of those ``born'' between $d$ and $b$ and ``still alive'' at $ v$. This is due to the presence of the connected component containing the point $s'( u')$.

On the other hand, since $R\cap \hat g(S^1)=\emptyset$
we have that
\begin{equation}\label{eq2}
    \mathrm{rk}\ H_0^{b, v}(\hat g)-\mathrm{rk}\ H_0^{d, v}(\hat g)-\mathrm{rk}\ H_0^{a, v}(\hat g)+\mathrm{rk}\ H_0^{c, v}(\hat g)=0.
\end{equation}

Indeed, with respect to the $C^1$ function $\hat g$, the number of connected components that are \lq\lq born'' between $c$ and $a$ and still not merged at $ v$ is equal to the number of those \lq\lq born'' between $d$ and $b$ and \lq\lq still alive'' at $ v$.
This fact contradicts the assumption that $H_0^{ u, v}(s\circ f)= H_0^{ u, v}(s\circ g)$ for every $( u, v)\in \Delta^+$ and every $s\in \Sigma_2$.

A formal proof of the equalities (\ref{eq1}) and (\ref{eq2}) will be given in the Appendix.

Therefore, we have proved that $f(\bar N_f)\subseteq g(S^1)$.
In the same way, we can prove that $g(\bar N_g)\subseteq f(S^1)$, where $\bar N_g$ is the closure of $N_g$ in $S^1$.

Now, let us prove that $f(S^1)\setminus f(\bar N_f)\subseteq g(S^1)$. Let us assume that $f(S^1)\setminus f(\bar N_f)\ne \emptyset$, otherwise the claim is trivial. Hence, let us take a point $u\in f(S^1)\setminus f(\bar N_f)$, and $\theta\in S^1$ such that $f(\theta)=u$. Let us consider the maximal open connected arc $\alpha$ in $S^1\setminus \bar N_f$ containing $\theta$. We shall prove that $f(\alpha)\subset g(S^1)$, which implies that $u\in g(S^1)$.

Because of the definition of $N_f$ and the regularity of $f$,  $f(\alpha)$ is   either a horizontal or a vertical segment. Let us assume that $f(\alpha)$ is a horizontal segment (the other case can be treated quite analogously). The arc $\alpha$ necessarily has two distinct endpoints $\theta^{in},\theta^{out}$ (listed counterclockwise). Possibly by changing $f$ into $\tilde f=s_1\circ f$, and $g$ into $\tilde g=s_1\circ g$, we can also assume that the horizontal segment $\alpha$ proceeds from left to right while the parameter $\theta$ increases.
We observe that  $N_f=N_{\tilde f}$ and $N_g=N_{\tilde g}$.

The points $f(\theta^{in}),f(\theta^{out})$ are the endpoints of the horizontal segment $f(\alpha)$. Since $f$ is $C^1$ we have that $\mathrm{im}\, d_{\theta^{in}}f$ is a horizontal line. Because of the maximality of $\alpha$,
we can find a sequence $(\theta_i)$ of points of $N_f$ converging counterclockwise to $\theta^{in}$. We already know that for each $\theta_i$ a point $\theta'_i$ exists such that $g(\theta'_i)=f(\theta_i)$. Possibly by extracting a subsequence, we can assume that $(\theta'_i)$ converges to a point $\theta'$. Since $f$ and $g$ are continuous, we get  $f(\theta^{in})=g(\theta')$. Furthermore, the equalities $g(\theta'_i)=f(\theta_i)$ imply that there is a sequence of coinciding incremental ratios for $f$ and $g$. Since $f$ and $g$ are $C^1$, we thus get that $\mathrm{im}\, d_{\theta^{in}}f=\mathrm{im}\, d_{\theta'}g$, and hence both of these lines are horizontal. Possibly by substituting $\theta$ with $-\theta$ as a parameter for $g$, we can also assume that the parallel and non-vanishing vectors
$\frac{df}{d\theta}(\theta^{in})$ and $\frac{dg}{d\theta}(\theta')$ have the same sense. We observe that the passage from $g(\theta)$ to $g(-\theta)$ does  change neither  $g(S^1)$ nor $N_g$.

Now we consider the last point $\theta^*$ in the closure of $\alpha$  (orienting $\alpha$ from $\theta^{in}$ to $\theta^{out}$) verifying the following property:
\begin{itemize}
  \item For every point $\bar\theta$ in the (possibly degenerate) closed arc from $\theta^{in}$ to $\theta^*$, a point $\bar\theta'$ exists for which $f(\bar\theta)=g(\bar\theta')$, and the non-vanishing vectors
$\frac{df}{d\theta}(\bar\theta)$ and $\frac{dg}{d\theta}(\bar\theta')$ are parallel and have the same sense.
\end{itemize}
We have seen that at least $\theta^{in}$ satisfies this property. We can prove that $\theta^*=\theta^{out}$. In order to do this, let us assume that $\theta^*\neq\theta^{out}$ and show that this implies a contradiction. Let $\theta'^*$ be a point in $S^1$ such that $f(\theta^*)=g(\theta'^*)$ and  $\frac{df}{d\theta}(\theta^*)$ and $\frac{dg}{d\theta}(\theta'^*)$ are parallel and have the same sense.  

In case there is no sequence of points of $N_g$ converging clockwise to $\theta'^*$, any sufficiently small open arc $\beta$ whose closure $\bar \beta$ contains $\theta'^*$ as a start point (with $\bar\beta$ counterclockwise oriented) is such that $\bar\beta\subset S^1\setminus\bar N_g$. By recalling that $\mathrm{im}\, d_{\theta'^*}g$ is a horizontal line, we get that $g(\bar\beta)$ is a horizontal segment. Since $\theta^*\neq\theta^{out}$ and $f(\alpha)$ is a horizontal segment, the point $g(\theta'^*)=f(\theta^*)$ does not equal $f(\theta^{out})$. Therefore, $g(\bar\beta)\subset f(\alpha)$ for any sufficiently small $\beta$, contradicting the definition of $\theta^*$.

Let us now consider the case when a sequence $(\theta'_i)$ of points of $N_g$ converges clockwise to $\theta'^*$. Because of the definition of the set $N_g$, possibly perturbing each point in the sequence, we can  assume that no point $g(\theta'_i)$ belongs to $f(\bar\alpha)$ (where $\bar \alpha$ denotes the closure of the arc $\alpha$).
We already know that, for each $\theta'_i$, a point $\hat\theta_i$ exists such that $f(\hat\theta_i)=g(\theta'_i)$. Possibly by extracting a subsequence, we can assume that $(\hat\theta_i)$ converges to a point $\hat\theta$. Since $f$ and $g$ are continuous, we get that $g(\theta'^*)=f(\hat\theta)$. Now, either $\hat\theta=\theta^*$ or $\hat\theta\ne \theta^*$.

Let $\hat\theta=\theta^*$.  Since $\theta^*\ne \theta^{out}$,   $(\hat\theta_i)$ converges to  $\hat\theta$ counterclockwise. We recall that  $\frac{df}{d\theta}( \theta^*)$ and $\frac{dg}{d\theta}( \theta'^*)$ are both non-vanishing horizontal vectors pointing to the right. This contradicts the fact that   $\frac{df}{d\theta}( \theta^*)=\lim_{i\to \infty}\frac{f(\hat\theta)-f(\hat\theta_i)}{\hat \theta-\hat\theta_i}=\lim_{i\to \infty}\frac{g(\theta'^*)-g(\theta'_i)}{\theta'^*-\theta'_i}\cdot \frac{\theta'^*-\theta'_i}{\hat\theta-\hat\theta_i} $, because $\lim_{i\to \infty}\frac{g(\theta'^*)-g(\theta'_i)}{\theta'^*-\theta'_i}=\frac{dg}{d\theta}( \theta'^*)$ and $\frac{\theta'^*-\theta'_i}{\hat\theta-\hat\theta_i}<0$ for every $i$ sufficiently large.

Now, let $\hat\theta\ne\theta^*$.  The equalities $f(\hat\theta_i)=g(\theta'_i)$ imply that there is a sequence of coinciding incremental ratios for $f$ and $g$. Since $f$ and $g$ are $C^1$, we thus get that $\mathrm{im}\, d_{\theta'^*}g=\mathrm{im}\, d_{\hat \theta}f$.
Now, the equalities $f(\theta^*)=g(\theta'^*)$ and $\mathrm{im}\, d_{\theta^*}f=\mathrm{im}\, d_{\theta'^*}g$ imply that $f(\theta^*)=f(\hat\theta)$ and $\mathrm{im}\, d_{\theta^*}f=\mathrm{im}\, d_{\hat\theta}f$, with $\hat\theta\neq \theta^*$. This contradicts the assumption that the double points of $f$ are clean (property (C2)).

Therefore, in any case the assumption $\theta^*\neq\theta^{out}$ implies a contradiction, so that it must be $\theta^*=\theta^{out}$. Hence the inclusion $f(\alpha)\subset g(S^1)$ is proven.

Therefore, we have proved that $f(S^1)\setminus f(\bar N_f)\subseteq g(S^1)$. In the same way, we can prove that $g(S^1)\setminus g(\bar N_g)\subseteq f(S^1)$.

In conclusion, we have proved that $f(S^1)= g(S^1)$.

Let us now construct the $C^1$-diffeomorphism $h:S^1\to S^1$ such that $g\circ h=f$. Since $f$ is generic, there is a finite set $\Theta_f=\{\theta^f_1,\ldots,\theta^f_r\}\subset S^1$ such that $f_{|S^1\setminus \Theta_f}$ is a $C^1$-diffeomorphism between $S^1\setminus \Theta_f$ and $f(S^1\setminus \Theta_f)$ (see properties (C1) and (C2) in Def.~\ref{d1}). Analogously, the genericity of $g$ implies that a finite set $\Theta_g=\{\theta^g_1,\ldots,\theta^g_s\}\subset S^1$ exists, such that $g_{|S^1\setminus \Theta_g}$ is a $C^1$-diffeomorphism between $S^1\setminus \Theta_g$ and $f(S^1\setminus \Theta_g)$.

Since $f(S^1)=g(S^1)$, for any $\theta\in S^1$ the set $g^{-1}\left(f(\theta)\right)$ is not empty. Moreover, since in particular $f(\Theta_f)=g(\Theta_g)$, if $\theta\in S^1\setminus \Theta_f$, the set $g^{-1}\left(f(\theta)\right)$ contains only one point $\theta'$ and we can define $h(\theta)=\theta'$. If $\theta\in \Theta_f$, we have that $g^{-1}\left(f(\theta)\right)=\{\theta'_1,\theta'_2\}$. In this case,  there is just one point $\theta'_i\in g^{-1}\left(f(\theta)\right)$ such that $\mathrm{im}\ dg_{\theta'_i}= \mathrm{im}\ df_{\theta}$, because, by property (C2) in Def.~\ref{d1}, double points of $g$ are clean. Thus, we can define $h(\theta)=\theta'_i$.

Because of its definition, the function $h$ verifies the equality $g\circ h=f$. We claim that $h$ is a $C^1$-diffeomorphism.
Indeed, recalling that  $f(S^1)=g(S^1)$, the definition of $h$ implies that $h$ is injective and surjective.
Furthermore, for each point $\theta\in S^1$,  there exist an open neighborhood $U(\theta)$ of $\theta$ in $S^1$ such that $f_{|U(\theta)}$ is a $C^1$-diffeomorphism,
 a point $\theta'\in S^1$ for which $g(\theta')=f(\theta)$, and
an open neighborhood $V(\theta')$ of $\theta'$ in $S^1$ such that $g_{|V(\theta')}$ is a $C^1$-diffeomorphism and  $g(V(\theta'))=f(U(\theta))$. Hence, $h_{|U(\theta)}$ equals the $C^1$-diffeomorphism $g^{-1}_{|V(\theta')}\circ f_{|U(\theta)}$. This concludes our proof.

\end{proof}

Incidentally, we observe that the proof of Theorem \ref{th1} could be simpler if we   asked generic functions to satisfy a further condition beside $(C1-2)$, that is
\begin{enumerate}
\item[(C3)] the set $\{\theta\in S^1: \frac{\ d f_1}{d\theta}(\theta)\ne 0 \ \wedge\  \frac{\ d f_2}{d\theta}(\theta)\ne 0\}$ is dense in $S^1$.
\end{enumerate}
Roughly speaking, (C3) says that, for almost every point, the tangent line to the curve is neither horizontal nor vertical.  This is still a generic property. Clearly, in this way, the proof that $f(S^1)\setminus f(\bar N_f)\subseteq g(S^1)$ would be trivial.  However, the price to pay would be some other complications in the next Theorem \ref{th2}.


From previous Theorem~\ref{th1} the next corollary follows:

\begin{cor}\label{th3}
Let $f,g:S^1\to \R^2$ be two continuous functions. If there exist two generic functions $f':S^1\to \mathbb{R}^2$, $g':S^1\to \mathbb{R}^2$ such that  $H_0^{ u, v}(s\circ f')= H_0^{ u, v}(s\circ g')$ for every $( u, v)\in \Delta^+$ and every $s\in \Sigma_2$, then  there exists a $C^1$-diffeomorphism $h:S^1\to S^1$ such that $\|g\circ h-f\|_\infty\le \|g-g'\|_\infty+ \|f'-f\|_\infty$.
\end{cor}

\begin{proof}
Theorem~\ref{th1} ensures that a $C^1$-diffeomorphism $h:S^1\to S^1$ exists such that $g'\circ h=f'$. As a consequence, $\|g\circ h-f\|_\infty\le \|g\circ h-g'\circ h\|_\infty+\|g'\circ h-f'\|_\infty + \|f'-f\|_\infty=\|g-g'\|_\infty+ \|f'-f\|_\infty$.
\end{proof}

Theorem \ref{th1} shows that persistent homology is sufficient to classify curves of $\R^2$ up to $C^1$-diffeomorphisms that preserve the considered functions, but this result seems not to be completely satisfactory. Indeed, in order to be applied, it requires complete coincidence of persistent homology groups, which may not occur in concrete applications.

In some sense the next result Theorem \ref{th2} improves Theorem \ref{th1}, since it translates our approach into a setting where it is requested only some kind of closeness between the persistent homology groups of the considered functions $f$, $g$, expressed by a suitable distance. In order to state Theorem \ref{th2}, we need to consider a restricted space of functions.

\begin{defn}\label{Fk}
For every positive real number $k$, we define $F_k$ to be the subset of $C^1(S^1,\R^2)$ such that

\begin{enumerate}
\item   $f$ is  generic;
\item $f(S^1)$ is contained in the disk of $\R^2$ centered at $(0,0)$ with radius $k$;
  \item $f$ is a curve of length $\ell_f$ with $\ell_f\le k$;
  \item The curvature of the curve $f$ is everywhere not greater than $k$;
  \item Every $C^1$ function $f':S^1\to\R^2$ such that $f'$ has a distance less than $\frac{1}{k}$ from $f$, with respect to the $C^1$-norm, is generic.
\end{enumerate}
\end{defn}

Let us recall that the set of generic functions is open in $C^1(S^1,\R^2)$ (see Proposition~\ref{generic}). Hence, for $k$ sufficiently large, the set $F_k$ is non-empty. Moreover, let us remark that,   for any generic $f$, there is a sufficiently large value $k(f)\in \R$ such that $f\in F_{k(f)}$.

\begin{lem}\label{Fkclosed}
The closure of $F_k$ in the $C^1$-topology is contained in the space of generic functions.
\end{lem}

\begin{proof}
Let $(f_i)$ be a   sequence in $F_k$, and assume $\lim_{i\to\infty}f_i$ exists and is equal to $\bar f$. Thus, the ball centered at $\bar f$ with radius $1/2k$ contains some function $f_i$. Since $f_i\in F_k$, by  property (v) in Definition \ref{Fk}, we see that $\bar f$ is generic.
\end{proof}

In order to measure the distance between the persistent Betti numbers functions $\mathrm{rk}\, H_0^{\cdot,\cdot}(f),\mathrm{rk}\, H_0^{\cdot,\cdot}(g):\Delta^+\to\mathbb{N}$, we use the matching distance $D_{match}$
defined and studied in \cite{BiCeFrGiLa08}. The main property of this distance (and the unique we use in this paper) is that it is stable with respect to perturbation of the functions. Indeed, the Multidimensional Stability Theorem in degree $0$ (see \cite[Thm. 4]{BiCeFrGiLa08}) states that
if $\|f-g\|_\infty\le \epsilon$ then $D_{match}\left(\mathrm{rk}\ H_0^{\cdot,\cdot}(f),\mathrm{rk}\ H_0^{\cdot,\cdot}(g)\right)\le \epsilon$. For the definition and the main results concerning this distance between the ranks of the persistent homology groups in degree $0$, i.e. size functions, we refer the interested reader to \cite{BiCeFrGiLa08}.

We can now extend Theorem~\ref{th1} to the following result.

\begin{thm}\label{th2}
Let $k>0$. For every $\epsilon>0$, a $\delta>0$ exists such that if $f,g\in F_k$ and the matching distance between the functions $\mathrm{rk}\, H_0^{ \cdot, \cdot}(s\circ f)$ and $\mathrm{rk}\, H_0^{ \cdot, \cdot}(s\circ g)$ is not greater than $\delta$ for every $s\in \Sigma_2$, then there exists a $C^1$-diffeomorphism $h:S^1\to S^1$ such that $\|f-g\circ h\|_\infty\le \epsilon$.
\end{thm}

\begin{proof}
Let us assume that our statement is false. Then a value $\bar \epsilon>0$ exists such that, for every $\delta>0$, two functions $f_\delta,g_\delta\in F_k$ exist, for which   $D_{match}\left(\mathrm{rk}\, H_0^{ \cdot, \cdot}(s\circ f), \mathrm{rk}\, H_0^{ \cdot, \cdot}(s\circ g)\right)\le  \delta$, for every $s\in \Sigma_2$, but $\|f_\delta-g_\delta\circ h\|_\infty> \bar\epsilon$, for every $C^1$-diffeomorphism $h:S^1\to S^1$. Since each persistent homology group is invariant by composition of the considered function with a homeomorphism, it is not restrictive to assume that, for every  $\delta>0$, the parameter $\theta$ is proportional to the arc-length parameter of the curves $f_\delta$ and $g_\delta$ (up to a shift).

It is easy to check that, because of our choice of the parameterization of $S^1$ and of the bounds assumed on the length and the curvature of the curves $f_\delta$ and $g_\delta$ (see properties (iii) and (iv) in the definition of $F_k$),  the first and second derivative of $f_\delta$ and $g_\delta$ are bounded by a constant independent of $\delta$. 


Let us consider the sequences $\left(f_{\frac{1}{i}}\right)$, $\left(g_{\frac{1}{i}}\right)$ ($i\in \mathbb{N}^+$).  Because of the definition of $F_k$, using the Ascoli-Arzel\`a Theorem (in its generalized version for higher derivatives, cf., e.g., \cite{Jo05}), and possibly extracting two subsequences, we can assume that $\left(f_{\frac{1}{i}}\right)$, $\left(g_{\frac{1}{i}}\right)$ converge to the $C^1$ functions $\bar f,\bar g$, , respectively, in the $C^1$-norm. Because of Lemma \ref{Fkclosed}, we know that $\bar f$ and $\bar g$ are generic. By applying the Multidimensional Stability Theorem in degree $0$ (cf. \cite[Thm. 4]{BiCeFrGiLa08}), we see that the matching distance between the functions $\mathrm{rk}\ H_0^{ \cdot, \cdot}(s\circ \bar f)$ and $\mathrm{rk}\ H_0^{ \cdot, \cdot}(s\circ \bar g)$ vanishes for every $s\in \Sigma_2$. Since $D_{match}$ is a distance, $\mathrm{rk}\, H_0^{ \cdot, \cdot}(s\circ \bar f)\equiv \mathrm{rk}\, H_0^{ \cdot, \cdot}(s\circ \bar g)$, and thus $H_0^{ u, v}(s\circ \bar f)=  H_0^{ u, v}(s\circ \bar g)$, for every $(u,v)\in \Delta^+$ and $s\in \Sigma_2$.

Therefore, we can apply Theorem~\ref{th1} and deduce that there exists a $C^1$-diffeomorphism $h:S^1\to S^1$ such that $\bar g\circ h=\bar f$.
As a consequence, $0=\|\bar g\circ h-\bar f\|_\infty= \lim_{i\to \infty} \|g_{\frac{1}{i}}\circ h-f_{\frac{1}{i}}\|_\infty \ge \bar\epsilon >0$. This is a contradiction, and hence our statement is proven.
\end{proof}

We conclude this paper by observing that
the presented approach  can be straightforwardly adapted to the case of curves in $\R^n$, and to the curves with more than one connected component. We leave the easy details to the reader. However, we note that generic curves in $\R^n$ with $n\ge 3$ have no multiple points.

The generalization of our results to surfaces seems to present some technical difficulties, and deserves a separate treatment.

\appendix
\section*{Appendix}
\setcounter{section}{1}

\begin{lem}\label{equalities}
Let  $R$ be an open rectangle in $\R^2$ with sides parallel to the coordinate axes and let $v=(v_1,v_2)$ be its top-right vertex. Let $a=(a_1,a_2)$, $b=(b_1,b_2)$, $c=(c_1,c_2)$, $d=(d_1,d_2)$  be four points in $R$ with $a_1=c_1<b_1=d_1<v_1$ and $c_2=d_2<a_2=b_2<v_2$ (see Fig. \ref{quad}). 
If $R\cap \hat f(S^1)=\emptyset$ then $$\mathrm{rk}\ H_0^{b, v}(\hat f)-\mathrm{rk}\ H_0^{d, v}(\hat f)-\mathrm{rk}\ H_0^{a, v}(\hat f)+\mathrm{rk}\ H_0^{c, v}(\hat f)=0.$$
If $\hat f$ is top-right transversal to $R$, then
$$\mathrm{rk}\ H_0^{b, v}(\hat f)-\mathrm{rk}\ H_0^{d, v}(\hat f)-\mathrm{rk}\ H_0^{a, v}(\hat f)+\mathrm{rk}\ H_0^{c, v}(\hat  f)=1.$$
\end{lem}

\begin{proof}
First of all we observe that $\mathrm{rk}\ H_0^{u, v}(\hat  f)$ is the number of connected components of $D^v$ that contain at least one point of $D^u$, by definition. So, letting  $n_u$ be  the number of connected components $C$ of $\hat f^{-1}(D^v)$ such that $\hat f(C)$ does not meet $D^u$, and $n$  the number of connected components of $\hat f^{-1}(D^v)$, $\mathrm{rk}\ H_0^{u, v}( \hat f)=n-n_u$.
As a consequence
$$\mathrm{rk}\ H_0^{b, v}(\hat  f)-\mathrm{rk}\ H_0^{d, v}(\hat f)-\mathrm{rk}\ H_0^{a, v}(\hat f)+\mathrm{rk}\ H_0^{c, v}(\hat f)=-n_b+n_d+n_a-n_c.$$

Let us consider the strips  $S^{top}_u=\{(x_1,x_2)\in D^v: x_2>u_2 \wedge (x_1,x_2)\not\in R\}$ and $S^{right}_u= \{(x_1,x_2)\in D^v : x_1>u_1 \wedge (x_1,x_2)\not\in R\}$, for $u=a,b,c,d$ (see Fig. \ref{figDprime}). We denote by $n^{top}_u$  the number of connected components $C$ of $\hat f^{-1}(D^v)$ such that $\hat f(C)$ is entirely contained in the strip $S^{top}_u$. Analogously, we denote by $n^{right}_u$  the number of connected components $C$ of $\hat f^{-1}(D^v)$ such that $\hat f(C)$ is entirely contained in the strip $S^{right}_u$.

If $R\cap \hat f(S^1)=\emptyset$, then $n_u=n^{top}_u+n^{right}_u$,  for $u=a,b,c,d$ (see Fig. \ref{figDprime}, top row). The equality $\mathrm{rk}\, H_0^{b, v}( \hat f)-\mathrm{rk}\, H_0^{d, v}(\hat  f)-\mathrm{rk}\, H_0^{a, v}(\hat f)+\mathrm{rk}\, H_0^{c, v}(\hat  f)=0$ follows by observing that $n^{top}_b=n^{top}_a$, $n^{top}_d=n^{top}_c$, $n^{right}_b=n^{right}_d$, $n^{right}_a=n^{right}_c$.

If $\hat f$ is top-right transversal to $R$, then $n_b=n^{top}_b+n^{right}_b$, $n_d=n^{top}_d+n^{right}_d+1$, $n_a=n^{top}_a+n^{right}_a+1$, $n_c=n^{top}_c+n^{right}_c+1$ (see Fig. \ref{figDprime}, bottom row). Once again, the equality $\mathrm{rk}\ H_0^{b, v}(\hat  f)-\mathrm{rk}\ H_0^{d, v}(\hat  f)-\mathrm{rk}\ H_0^{a, v}(\hat f)+\mathrm{rk}\ H_0^{c, v}(\hat  f)=1$ follows by observing that $n^{top}_b=n^{top}_a$, $n^{top}_d=n^{top}_c$, $n^{right}_b=n^{right}_d$, $n^{right}_a=n^{right}_c$.

\begin{figure}[h]
\begin{center}
\begin{tabular}{c}
\psfrag{R}{$R$} \psfrag{a}{$a$} \psfrag{b}{$b$}\psfrag{c}{$c$}\psfrag{d}{$d$}\psfrag{v}{$v$}\psfrag{D1}{$D^v_d\setminus R$}\psfrag{D2}{$D^v_a\setminus R$} \psfrag{D3}{$D^v_b\setminus R$}\psfrag{D4}{$D^v_c\setminus R$}
\psfrag{Sat}{$S^{top}_a$}\psfrag{Sbt}{$S^{top}_b$}\psfrag{Sct}{$S^{top}_c$}\psfrag{Sdt}{$S^{top}_d$}
\psfrag{Sar}{$S^{right}_a$}\psfrag{Sbr}{$S^{right}_b$}\psfrag{Scr}{$S^{right}_c$}\psfrag{Sdr}{$S^{right}_d$}
\includegraphics[width=14cm]{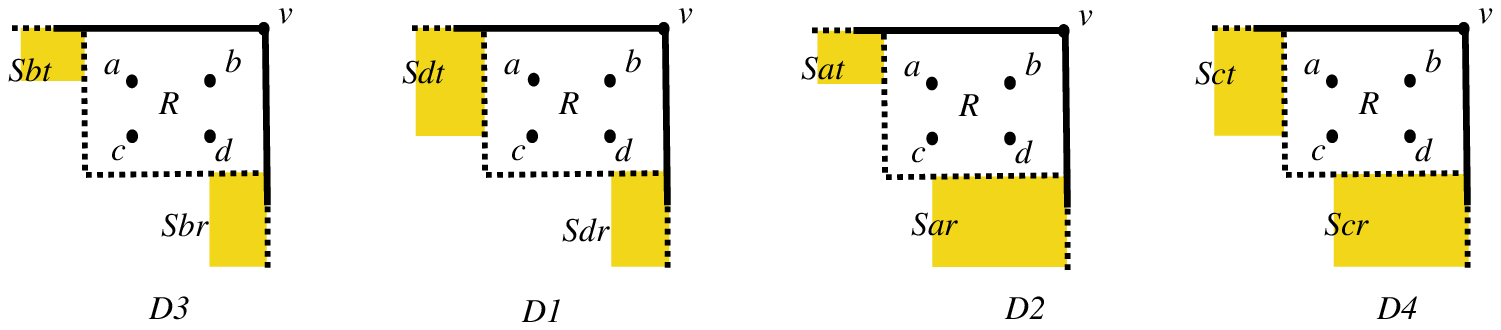}\\
\psfrag{R}{$R$} \psfrag{a}{$a$} \psfrag{b}{$b$}\psfrag{c}{$c$}\psfrag{d}{$d$}\psfrag{v}{$v$}\psfrag{D1}{$\hat f(\gamma)\subseteq D^v_d$}\psfrag{D2}{$\hat f(\gamma)\subseteq D^v_a$} \psfrag{D3}{$\hat f(\gamma)\not \subseteq D^v_b$}\psfrag{D4}{$\hat f(\gamma)\subseteq D^v_c$}\psfrag{G}{$ $}
\includegraphics[width=14cm]{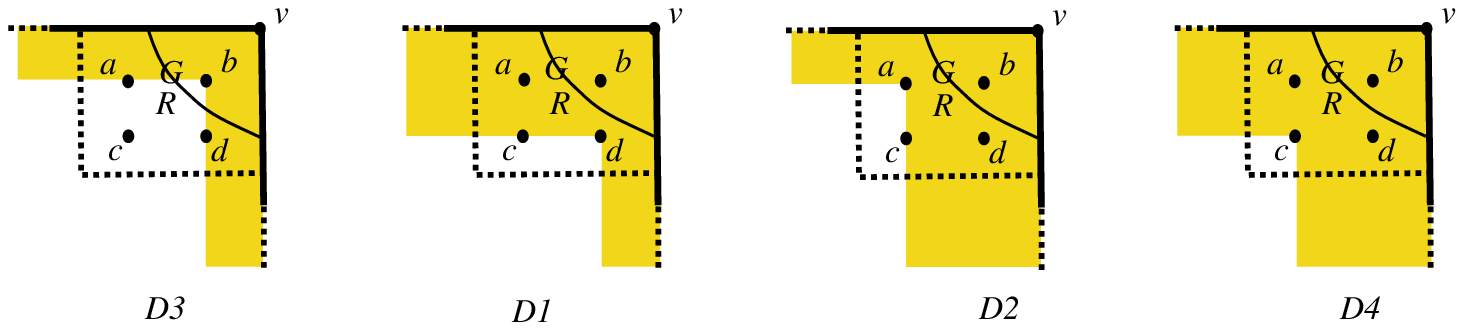}
\end{tabular}
\caption{The  colored sets in the top row correspond to the sets $D^v_a\setminus R$, $D^v_b\setminus R$, $D^v_c\setminus R$, and $D^v_d\setminus R$. Observe that each one of these sets contains two disjoint strips, and each strip is contained in exactly two of those sets.
The  colored sets in the bottom row correspond to the sets $D^v_a$, $D^v_b$, $D^v_c$, and $D^v_d$. Top row:  if $R\cap \hat f(S^1)=\emptyset$, then $-n_b+n_d+n_a-n_c=0$. Bottom row: if $R\cap \hat f(S^1)=\hat f(\gamma)$, then $-n_b+n_d+n_a-n_c=1$.}\label{figDprime}
\end{center}
\end{figure}

\end{proof}

\end{document}